\documentclass[11pt, leqno, equal margin]{article} \textwidth=14cm
\usepackage{amsmath,amsthm}
\usepackage{amssymb}

\usepackage{enumerate}

\usepackage{graphicx}

\newtheorem{thm}{Theorem}[section]
\newtheorem{cor}[thm]{Corollary}
\newtheorem{lem}[thm]{Lemma}
\newtheorem{prop}[thm]{Proposition}

\theoremstyle{definition}
\newtheorem{defin}[thm]{Definition}
\newtheorem{exa}[thm]{Example}

\newtheorem{rem}[thm]{Remark}

\begin{document}

\title{Local dimension theory of tensor products of algebras over a ring}

\author{Samir Bouchiba \footnote {{\it Email}: bouchibasamir@gmail.com}\\
{\footnotesize Department of Mathematics, Faculty of Sciences,}
{\footnotesize University Moulay Ismail,}\\ {\footnotesize Meknes,
Morocco}}
\date{}

\maketitle


\renewcommand{\thefootnote}{}

\footnote{2010 \emph{Mathematics Subject Classification}: 13C15;
13A15.}

\footnote{\emph{Key words and phrases}: Krull dimension; fibre ring;
AF-ring; Fibred AF-ring; height; Boolean ring.}

\renewcommand{\thefootnote}{\arabic{footnote}}
\setcounter{footnote}{0}

\begin{abstract} Our main goal in this paper is to set the general frame for studying the dimension theory of
tensor products of algebras over an arbitrary ring $R$. Actually, we
translate the theory initiated by A. Grothendieck and R. Sharp and
subsequently developed by A. Wadsworth on Krull dimension of tensor
products of algebras over a field $k$ into the general setting of
algebras over an arbitrary ring $R$. For this sake, we introduce and
study the notion of a fibred AF-ring over a ring $R$. This concept
extends naturally the notion of AF-ring over a field introduced by
A. Wadsworth in \cite{W} to algebras over arbitrary rings. We prove
that Wadsworth theorems express local properties related to the
fibre rings of tensor products of algebras over a ring. Also, given
a triplet of rings $(R,A,B)$ consisting of two $R$-algebras $A$ and
$B$ such that $A\otimes_RB\neq \{0\}$, we introduce the inherent
notion to $(R,A,B)$ of a $B$-fibred AF-ring which allows to compute
the Krull dimension of all fiber rings of the considered tensor
product $A\otimes_RB$. As an application, we provide a formula for
the Krull dimension of $A\otimes_RB$ when $A$ and $B$ are
$R$-algebras with $A$ is zero-dimensional as well as for the Krull
dimension of $A\otimes_{\mathbb{Z}}B$ when $A$ is a fibred AF-ring
over the ring of integers $\mathbb{Z}$ with nonzero characteristic
and $B$ is an arbitrary ring. This enables us to answer a question
of Jorge Matinez on evaluating the Krull dimension of
$A\otimes_{\mathbb{Z}}B$ when $A$ is a Boolean ring. Actually, we
prove that if $A$ and $B$ are rings such that
$A\otimes_{\mathbb{Z}}B$ is not trivial and $A$ is a Boolean ring,
then dim$(A\otimes_{\mathbb{Z}}B)=\mbox {dim}\Big (\displaystyle
{\frac B{2B}}\Big )$.\end{abstract}

\noindent {\footnotesize {\it MSC} (2000): 13D02; 13D05; 13D07; 16E05; 16E10.}\\

\section{Introduction}
All rings considered in this paper are commutative with identity
element and all ring homomorphisms are unital. Here and
subsequently, $R$ stands for an arbitrary ring and $k$ stands for a
field. Let $A$ be a ring and $p$ be a prime ideal of $A$. Then
Spec($A)$ denotes the set of all prime ideals of $A$ and $k_A(p)$
denotes the quotient field of $\displaystyle {\frac Ap}$. Also, if
$n\geq 0$ is a positive integer, $A[n]$ denotes the polynomial ring
in $n$ indeterminates $A[X_1,X_2,\cdots,X_n]$ and $ht(p[n])$ stands
for the height of the extended ideal $p[n]:=p[X_1,X_2,\cdots,X_n]$
of $p$. Further, if $A$ is an algebra over a field $k$, then
t.d.$(A:k)$ denotes the transcendence degree of $A$ over $k$.
 Any
unreferenced material is standard as in \cite{M}, \cite{Bo} and
\cite{ZS}.

It is a paper of R. Sharp on Krull dimension of tensor products of
two field extensions of a field $k$ which gave the initial impetus
to study the Krull dimension of tensor products. Actually, in
\cite{S}, Sharp proved that, for any two extension fields $K$ and
$L$ of $k$, dim$(K\otimes_kL)=$ min(t.d.($K:k),$ t.d.($L:k))$
(actually, this result appeared ten years earlier in Grothendieck's
EGA \cite[Remarque 4.2.1.4, p. 349]{EGA4}). This formula is rather
surprising since, as one may expect, the structure of the tensor
product should reflect the way the two components interact and not
only the structure of each component. This fact affords motivation
to Wadsworth to work on this subject in \cite{W}. He aimed at
seeking geometric properties of primes of $A\otimes_kB$ and to widen
the scope of $k$-algebras $A$ and $B$ whose tensor product Krull
dimension, dim$(A\otimes_kB)$, shows exclusive dependence on
individual characteristics of $A$ and $B$. The algebras which proved
to be tractable for Krull dimension computations turned out to be
those rings $A$ which satisfy the altitude formula over $k$
(AF-rings for short), that is,
$$ht(p)+\mbox {t.d.}\Big (\displaystyle {\frac Ap}:k\Big )=\mbox { t.d.}(A_p:k)$$ for all prime
ideals $p$ of $A$. The class of AF-rings contains the most basic
rings of algebraic geometry, including finitely generated
$k$-algebras. Wadsworth proved through \cite[Theorem 3.7]{W} that if
$A$ is an AF-domain and $B$ is any $k$-algebra, then
$$\mbox {dim}(A\otimes_kB)=\mbox { max}\Big \{ht(q[\mbox
{t.d.}(A:k)])+\mbox {min}\Big (\mbox {t.d.}(A:k),ht(p)+\mbox
{t.d.}\Big (\displaystyle {\frac Bq}:k\Big )\Big ):$$ $$p\in\mbox {
Spec}(A)\mbox { and }q\in\mbox { Spec}(B)\Big \}.$$ As a consequence
of this, \cite[Theorem 3.8]{W} states that if $A_1$ and $A_2$ are
both AF-domains, then
$$\mbox {dim}(A_1\otimes_kA_2)=\mbox { min}\Big {(}\mbox
{dim}(A_1)+\mbox {t.d.}(A_2:k),\mbox { t.d.}(A_1:k)+\mbox
{dim}(A_2)\Big {)}.$$ Further, he gave a result which yields a
classification of the prime ideals of $A\otimes_kB$ according to
their contractions to $A$ and $B$ (cf. \cite[Proposition 2.3]{W}).
In \cite{BGK1}, we continued the work of Wadsworth and transferred
all his theorems in \cite{W} on AF-domains to AF-rings. This passage
from domains to rings with zero-divisors is well reflected in new
formulas for the Krull dimension of tensor products involving
AF-rings. As it turns out from the present work, it is these
formulas that are relevant in our treatment of the general setting
of tensor products over a ring $R$. We refer the reader to
\cite{BGK1,BGK2,BDK,BK,B1,B2,S,S1,W} for basics and recent
investigations on the dimension theory of tensor products of
algebras over a field.

The main goal of this paper is to set the general frame to study the
dimension theory of tensor products of algebras over an arbitrary
ring $R$. Actually, we translate the theory initiated by A.
Grothendieck and R. Sharp and subsequently developed by A. Wadsworth
on Krull dimension of tensor products of algebras over a field $k$
into the general setting of algebras over an arbitrary ring $R$. It
turns out that Wadsworth theorem express local properties related to
the fibre rings of the tensor products over an arbitrary ring $R$.
For this sake, we
 introduce and study the notion of
a fibred AF-ring over a ring $R$. Actually, we say that an
$R$-algebra $A$ is a fibred AF-ring over $R$ if the fibre ring
$k_R(p)\otimes_RA$ is an AF-ring over $k_R(p)$ for any prime ideal
$p$ of $R$ such that $k_R(p)\otimes_RA\neq \{0\}$. When restricted
to tensor products over a field the notion of a fibred AF-ring boils
down to the classical one of an AF-ring. It is notable that all
finitely generated algebras over $R$ proved to be fibred AF-rings as
well as all zero-dimensional rings which are $R$-algebras. We prove
that the fibred AF-rings inherit all properties of Wadsworth
introduced AF-rings. Moreover, given a triplet of rings $(R,A,B)$
consisting of $R$-algebras $A$ and $B$ such that $A\otimes_RB\neq
\{0\}$, we introduce and study the notion of a $B$-fibred AF-ring
over $R$ which is a somewhat inherent concept to the given triplet
$(R,A,B)$. So, when $A$ is a $B$-fibred AF-ring over $R$, we can
explicit the Krull dimension of all fiber rings of $A\otimes_RB$ and
in various cases this enables us to determine dim$(A\otimes_RB)$. As
an application, we compute the Krull dimension of $A\otimes_RB$ when
$A$ and $B$ are $R$-algebras with $A$ is zero-dimensional. Also, we
provide a formula for the Krull dimension of $A\otimes_RB$ when $A$
and $B$ are $R$-algebras with $A$ is zero-dimensional as well as for
the Krull dimension of $A\otimes_{\mathbb{Z}}B$ when $A$ is a fibred
AF-ring over the ring of integers $\mathbb{Z}$ with nonzero
characteristic and $B$ is an arbitrary ring. This allows us to
answer a question of Jorge Matinez on evaluating the Krull dimension
of $A\otimes_{\mathbb{Z}}B$ when $A$ is a Boolean ring. Actually, we
prove that if $A$ and $B$ are rings such that
$A\otimes_{\mathbb{Z}}B$ is not trivial and $A$ is a Boolean ring,
then dim$(A\otimes_{\mathbb{Z}}B)=\mbox {dim}\Big
(\displaystyle {\frac B{2B}}\Big )$.\\

\noindent {\bf Acknowledgement}\\
I would like to thank Professor Jeorge Martinez for the discussion
led with him on the possible connections between the dimension
theory of tensor products of algebras and the dimension theory of
frames. His questions on this issue were the source of motivation to
write this paper.

\section{Local spectrum and effective spectrum}

This section introduces the effective spectrum of a ring $R$ with
respect to an $R$-algebra as well as local notions of well known
concepts of dimension theory of rings such that the height of a
prime ideal, the Krull dimension and the spectrum of a ring.

Let $R$ be an arbitrary ring and let $A$ be an $R$-algebra. Denote
by $f_A:R\longrightarrow A$, with $f_A(r)=r.1_A$ for any $r\in R$
and where $1_A$ is the unit element of $A$, the ring homomorphism
defining the structure of algebra of $A$ over $R$. Let $K:=$
Ker$(f_A)$. It is easily seen that for each prime ideal $P$ of $A$,
$K\subseteq p:=f_A^{-1}(P)$ and that the induced homomorphisms
$\overline {f_A}:\displaystyle {\frac RK}\longrightarrow A$ and
$\widetilde {f_A}:\displaystyle {\frac Rp\longrightarrow \frac AP},$
defined by $\overline {f_A}(\overline r)=f_A(r)$ and $\widetilde
f_A(\tilde r)=\widetilde {f_A(r)}$ for each $r\in R$, are injective.
Let $S$ be a multiplicative subset of $R$. Recall that the
localization of $A$ by $S$ is the $S^{-1}R$-algebra
$S^{-1}A:=S^{-1}R\otimes_RA$. Our first result proves that $S^{-1}A$
is isomorphic to a localization of $A$ by a multiplicative subset
$\overline S$ of $A$. By virtue of this lemma, we deduce that, for
any multiplicative subset $S$ of $R$, $S^{-1}A$ enjoys all
properties satisfied by the well known localization by a
multiplicative subset of $A$.

\begin{lem} Let $R$ be a ring and $A$ an $R$-algebra. Let $S$ be a
multiplicative subset of $R$. Let $\overline S:=f_A(S)$ denote the
the corresponding multiplicative subset of $A$. Then the natural map
$\varphi:S^{-1}A\longrightarrow \overline S^{-1}A$ such that
$\displaystyle {\varphi\Big (\sum\limits_{i\in\Lambda}\frac
{r_i}{s_i}\otimes_Ra_i\Big )=\sum\limits_{i\in \Lambda}\frac
{r_ia_i}{f_A(s_i)}}$, for any finite set $\Lambda$, any $\{r_i:i\in
\Lambda\}\subseteq R$, $\{s_i:i\in \Lambda\}\subseteq S$ and
$\{a_i:i\in\Lambda\}\subseteq A$, is an isomorphism of
$R$-algebras.\end{lem}

\begin{proof} It is easy to see that the mapping $f:S^{-1}R\times
A\longrightarrow \overline S^{-1}A$ such that $f\Big (\displaystyle
{\frac rs,a\Big )=\frac {ra}{f_A(s)}}$ is well defined and is
$R$-biadditive. This yields the existence of the assigned
homomorphism $\varphi$ of $R$-algebras. Also, it is routine to check
that the map $g:\overline S^{-1}A\longrightarrow S^{-1}R\otimes_RA$
defined by $g\Big (\displaystyle {\frac a{f_A(s)}\Big ):=g\Big
(\frac a{s.1_A}\Big )=\frac 1s\otimes_Ra}$ for each $a\in A$ and
each $s\in S$ is an homomorphism of $R$-algebras. Then observe that
$\varphi\circ g=$ id$_{\overline S^{-1}A}$ and $g\circ\varphi =$
id$_{S^{-1}A}$. Hence $\varphi$ is an isomorphism of $R$-algebras.
\end{proof}

The above discussion allows us to announce the following lemma which
collects certain properties and facts about fibre rings. These
properties were stated in \cite[page 84]{MS} in the Noetherian
setting but in fact they hold in the general case.\\

\begin{lem} Let $R$ be a ring and let $A$ be an $R$-algebra. Let
$p$ be a prime ideal of $R$. Let $
S_p:=\displaystyle {\frac Rp}\setminus \{\overline 0\}$. Then\\
1) $\displaystyle {k_R(p)\otimes_RA\cong S_p^{ -1}\frac A{pA}:=S_p^{
-1}\frac A{f_A(p)A}}$ and $\displaystyle {k_R(p)\otimes_RP\cong
S_p^{-1}\frac P{pA}}$ for each prime ideal $P$ of $A$ such that $f_A^{-1}(P)=p$.\\
2) Spec$(k_R(p)\otimes_RA)=\{k_R(p)\otimes_RP:P\in$ Spec$(A)$ such that $f_A^{-1}(P)=p\}$.\\
3) There exists an order-preserving bijective correspondence between
the spectrum of $k_R(p)\otimes_RA$ and the set of prime ideals of
$A$ which contract to $p$ over $R$.\\
4) Let $P$ be a prime ideal of $A$ and let $p:=f_A^{-1}(P)$. Then
$$(k_R(p)\otimes_RA)_{k_R(p)\otimes_RP}\cong k_R(p)\otimes_RA_P.$$
5) Let $P$ be a prime ideal of $A$ and let $p:=f_A^{-1}(P)$. Then
$$\displaystyle {\frac {k_R(p)\otimes_RA}{k_R(p)\otimes_RP}\cong
k_R(p)\otimes_R\frac AP}.$$
\end{lem}

\begin{proof} In view of \cite[p. 84]{MS}, it remains to give a proof of (4) and (5).\\
4) Let $P$ be a prime ideal of $A$ and $p:=f_A^{-1}(P)$. Consider
the multiplicative subset $T:=R\setminus p$ of $R$. Then, by (1),
$$\begin{array}{lll}(k_R(p)\otimes_RA)_{k_R(p)\otimes_RP}&\cong&\displaystyle {\Big (S_p^{ -1}\frac
A{pA}\Big )_{S_p^{ -1}\frac P{pA}}}\\
&\cong&\displaystyle {\Big (\frac A{pA}\Big )_{\frac P{pA}}}\\
&\cong&\displaystyle {\frac {A_P}{pA_P}}\\
&\cong& \displaystyle {\frac Rp}\otimes_RA_P.\end{array}$$ Also,
notice that, on the one hand,
$$\begin{array}{lll}T^{-1}R\otimes_R\displaystyle {\frac
Rp}\otimes_RA_P&\cong&T^{-1}\displaystyle {\frac Rp}\otimes_RA_P\\
&=&S_p^{-1}\displaystyle {\frac Rp}\otimes_RA_P\\
&=&k_R(p)\otimes_RA_P\end{array}$$ and, on the other,
$$\begin{array}{lll}T^{-1}R\otimes_R\displaystyle {\frac
Rp}\otimes_RA_P&\cong&\displaystyle {\frac
Rp}\otimes_R(T^{-1}R\otimes_RA_P)\\
&\cong&\displaystyle {\frac Rp}\otimes_RT^{-1}A_P\\
&=&\displaystyle {\frac Rp}\otimes_RA_P \mbox { as } f_A(T)\subseteq
A\setminus P \mbox { and thus each element of }\\
&&f_A(T)\mbox { is invertible in }A_P.\end{array}$$ It follows that
$(k_R(p)\otimes_RA)_{k_R(p)\otimes_RP}\cong k_R(p)\otimes_RA_P$, as
desired.\\
5) Let $P$ be a prime ideal of $A$ and $p:=f_A^{-1}(P)$. Then, by
(1),$$\begin{array}{lll}\displaystyle {\frac
{k_R(p)\otimes_RA}{k_R(p)\otimes_RP}}&\cong&\displaystyle {\frac
{S_p^{
-1}(A/pA)}{S_p^{ -1}(P/pA)}}\\
&\cong&S_p^{-1}\displaystyle {\frac AP}\\
&\cong&k_R(p)\otimes_R\displaystyle {\frac AP} \mbox { as }p\mbox {
}\displaystyle {\frac AP}=(\overline 0).\end{array}$$ This completes
the proof.\end{proof}

Given an $R$-algebra $A$, it is to be noted that not all the prime
ideals of $R$ are essential in capturing the prime ideal structure
of $A$ over $R$. Actually, there are prime ideals and chains of
prime ideals of $R$ that have no effect on the structure of the
spectrum of $A$ (see Example 2.5). That is the reason why we
introduce in what follows the notion of effective spectrum of $R$
with respect to $A$ and effective Krull dimension of $R$ with
respect to $A$.

\begin{defin} {\it Let $A$ be an $R$-algebra.}\\
1) {\it A prime ideal $p$ of $R$ is said to be an effective prime
ideal of $R$ with respect to $A$ if the fibre ring
$k_R(p)\otimes_RA\neq
\{0\}$.}\\
2) {\it We define the effective spectrum of $R$ with respect to $A$
to be the set denoted by Spec$_{e}^A(R)$ consisting of all effective
prime ideals $p$ of $R$ with respect to $A$, namely,
$$\mbox { Spec}_e^A(R)=\{p\in\mbox { Spec}(R):k_R(p)\otimes_RA\neq
\{0\}\}.$$  Also, we denote by Max$_e^A(R)$ the subset of maximal
elements of Spec$_e^A(R)$, that is, the set of maximal effective
prime
ideals of $R$ with respect to $A$.}\\
3) {\it Let $p\in$ Spec$_e^A(R)$. We define the effective height of
$p$ with respect to $A$, denoted by $ht_e^A(p)$, to be the supremum
of lengths of chains of effective prime ideals
$p_0\subset p_1\subset\cdots\subset p_n=p$ of $R$ terminating at $p$.}\\
4) {\it We define the effective Krull dimension of $R$ with respect
to $A$ to be the invariant denoted by $\mbox {dim}_e^A(R)$ which is
the supremum of effective heights of effective prime ideals of $R$
with respect to $A$, that is,}
$$\mbox {dim}_e^A(R)=\mbox {sup}\{ht_e^A(p):p\in\mbox {Spec}_e^A(R)\}.$$.\end{defin}

The following result determines the effective spectrum of a ring $R$
with respect to various constructions. Its proof is easy and thus
omitted.

\begin{prop} 1) Let $A$ be an $R$-algebra. Then Spec$_e^A(R)=f_A^{-1}($Spec$(A))$\\
2) Let $X_1,X_2,\cdots,X_n$ be indeterminates over $R$. Then
Spec$_e^{R[X_1,X_2,\cdots,X_n]}(R)=$ Spec$(R)$.\\
3) Let $R\hookrightarrow A$ be an integral extension of rings. Then
Spec$_e^A(R)=$ Spec$(R)$.\\
4) Let $S$ be a multiplicative subset of $R$. Then
Spec$_e^{S^{-1}R}(R)=\{p\in$ Spec$(R):p\cap S=\emptyset\}$.\\
5) Let $A$ be an $R$-algebra and $S$ be a multiplicative subset of
$A$. Then $$\mbox {Spec}_e^{S^{-1}A}(R)=f_A^{-1}\Big (\{P\in\mbox {
Spec}(A):P\cap S=\emptyset\}\Big )=f_A^{-1}(\mbox
{Spec}_e^{S^{-1}A}(A))\subseteq\mbox { Spec}_e^{A}(R).$$ 6) Let $I$
be an ideal of $R$. Then Spec$_e^{\frac RI}(R)=\{p\in$
Spec$(R):I\subseteq p\}$.\\\end{prop}

Next, for any positive integer $n\geq 2$, we exhibit an example of a
ring $R$ and an $R$-algebra $A$ such that there exists a chain of
distinct prime ideals $p_0\subset p_1\subset\cdots\subset p_n$ in
$R$ with both ends $p_0,p_n\in$ Spec$_e^A(R)$ while the intermediate
elements $p_2,\cdots,p_{n-1}\not\in$ Spec$_e^A(R)$. It would be
interesting to afford a ring $T$
issued from $R$ and $A$ such that Spec$(T)=$ Spec$_e^A(R)$. This task is not easy and it turns out from the next example
that $T$ is neither a localization of $R$ nor a factor ring of $R$. \\

\begin{exa} Let $k$ be a field and $t$ be an
indeterminate over $k$. It is known, by \cite[Lemma 1]{MS}, that
there exists an infinite number of formal power series
$g_1(t),g_2(t),\cdots,$ $g_m(t),\cdots$ of $k[[t]]$ which are
algebraically independent over $k$. In fact, we can choose the $g_i$
in the maximal ideal $tk[[t]]$. Actually, assume that
$g_1(t)=1+a_1t+a_2t^2+\cdots+a_mt^m+\cdots\in k[[t]]\setminus
tk[[t]]$. Then, observe that $h_1(t):=g_1(t)-1\in tk[[t]]$ and, for
each integer $i\geq 2$, $h_i(t):=g_i(t)-g_i(0)g_1(t)\in tk[[t]]$,
and the formal power series $h_1(t),h_2(t),\cdots,h_m(t),\cdots$ are
algebraically independent over $k$ as
$k[h_1(t),h_{i_1}(t),h_{i_2}(t),\cdots,h_{i_n}(t)]=k[g_1(t),g_{i_1}(t),g_{i_2}(t),\cdots,g_{i_n}(t)]$
is a polynomial ring in $n+1$ indeterminates for any finite subset
$\{i_1,i_2,\cdots,i_n\}$ of $\mathbb{N}\setminus\{0\}$. Therefore
let $g_1(t),g_2(t),\cdots,g_m(t),\cdots\in tk[[t]]$ be algebraically
independent elements over $k$. Let $n\geq 2$ be an integer and
$X_1,X_2,\cdots,X_n$ be indeterminates over $k$. Let
$R:=k[X_1,X_2,\cdots,X_n]$ and $A:=k[[t]]$ be the formal power
series ring over $k$ which is a rank one discrete valuation ring and
thus Spec$(A)=\{(0),tk[[t]]\}$. We endow $A$ with the $R$-algebra
structure induced by the ring homomorphism $f_A:R\longrightarrow A$
such that $f_A(X_i)=g_i(t)$ for each $i=1,2,\cdots,n$. Therefore
$f_A^{-1}(tk[[t]])=(X_1,X_2,\cdots,X_n)$ is a maximal ideal of $R$
of height $n$. Also, as $g_1,g_2,\cdots,g_n$ are algebraically
independent over $k$, it is readily checked that $f_A$ is injective.
Hence $f_A^{-1}((0))=(0)$. It follows that
Spec$_e^A(R)=\{(0),(X_1,X_2,\cdots,X_n)\}$ and, in particular, any
prime ideal of $R$ properly between $(0)$ and the maximal ideal
$(X_1,X_2,\cdots,X_n)$ is not effective with respect to $A$, as
desired .\end{exa}

\begin{rem} 1) Let $A$ be an $R$-algebra. Then, for any $A$-algebra
$B$,  the natural map $f_A^{-1}:$ Spec$_e^B(A)\longrightarrow$
Spec$_e^B(R)$ is surjective while, in general, $f_A^{-1}:$ Spec$(A)\longrightarrow$ Spec$(R)$ is not so.\\
2) Let $A$ be an $R$-algebra. Then dim$_e^A(R)\leq$ dim$(R)$ and
this inequality might be strict as proved by Example 2.5 which shows
that for any positive integer $n\geq 2$ there exists a ring $R$ and
an $R$-algebra $A$ such that dim$(R)=n$ while dim$_e^A(R)=1$, as
desired.
\end{rem}

To get prepared for the general setting of tensor products, we next
introduce local notions of well known concepts of the dimension
theory of rings, namely the height of a prime ideal of a ring $A$,
the spectrum of $A$ and the Krull dimension of $A$.

\begin{defin} {\it Let $R$ be a ring and $A$ be an $R$-algebra.\\
1) Let $p\in$ Spec$_e^A(R)$. Then}

a) {\it Spec$_p(A):=\{I\in$ Spec$(A):f_A^{-1}(I)=p\}$ denotes the
set of all prime ideals of $A$ which contract to $p$ over $R$}.

b) {\it If $I\in$ Spec$_p(A)$, then the height of $I$ at $p$,
denoted by $ht_p(I)$, is the maximum of lengths of chains
$I_0\subset I_1\subset\cdots\subset I_n=I$ of prime ideals of $A$
which contract to $p$ over $R$.}

c) {\it The Krull dimension of $A$ at $p$ is the invariant
$$\mbox {dim}_p(A):=\mbox
{max}\{ht_p(I):I\in\mbox { Spec}_p(A)\}.$$} 2) {\it We define the
fibre Krull dimension of $A$ with respect to $R$ to be the maximal
length of chains of prime ideals of $A$ lying over a common
(effective) prime ideal of $R$ (with respect to $A$), that is the
invariant
$$\mbox {f-dim}_R(A)=\mbox { sup}\{\mbox {dim}_p(A):p\in\mbox {
Spec}_e^A(R)\}.$$} \end{defin}

By virtue of Lemma 2.2, we get the following result which connects
the above local data of $A$ with those relative to fibre rings
issued from $A$.

\begin{cor} Let $R$ be a ring and $A$ be an $R$-algebra. Let $p\in$ Spec$(R)$. Then\\
1) There exists an order-preserving bijective correspondence between
Spec$(k_R(p)\otimes_RA)$ and Spec$_p(A)$.\\
2) If $I\in$ Spec$_p(A)$, then
$ht_p(I)=ht(k_R(p)\otimes_RI)$.\\
3) dim$_p(A)=$ dim$(k_R(p)\otimes_RA)$.\\
4) f-dim$_R(A)=$ sup$\{$dim$(k_A(p)\otimes_RA):p\in$
Spec$_e^A(R)\}$.\end{cor}

Next, given an $R$-algebra $A$, we give lower and upper bounds of
the Krull dimension of $A$ in terms of the Krull dimension of its
fibre rings and the effective Krull dimension of $R$ with respect to
$A$. Observe that the formulas given in the following theorem are
reminiscent of Seidenberg's inequalities for the Krull dimension of
polynomial rings.

Recall that a ring homomorphism $f:A\longrightarrow B$ is said to
satisfy the Going-Down property (GD for short) if for any prime
ideals $p\subseteq q$ of $A$ such that there exists $Q\in$ Spec$(B)$
with $Q\cap A=q$, then there exists $P\in$ Spec$(B)$ such that
$P\cap A=p$ and  $P\subseteq Q$. It is then easy to see that if a
ring homomorphism $f:R\longrightarrow A$ satisfies GD and if $p\in$
Spec$_e^A(R)$, then any prime ideal $q$ of $R$ such that $q\subseteq
p$ is an effective prime ideal of $R$ with respect to $A$, and thus
ht$_e^A(p)=$ ht$(p)$.

\begin{thm} Let $R$ be a ring and let $A$ be an $R$-algebra. Then\\
1) $\mbox {f-dim}_R(A)\leq\mbox {dim}(A)\leq \mbox
{f-dim}_R(A)+(1+\mbox {f-dim}_R(A))\mbox {dim}_e^{A}(R).$\\
2) If the homomorphism $f_A:R\longrightarrow A$ satisfies the
Going-down property, then,
$$\mbox {sup}\{ht(p)+\mbox {dim}_p(A):p\in\mbox {Spec}_e^A(R)\}\leq\mbox {dim}(A)\leq
\mbox {f-dim}_R(A)+(1+\mbox {f-dim}_R(A))\mbox
{dim}_e^{A}(R).$$\end{thm}

\begin{proof} 1) It suffices to prove the second inequality. If either dim$_e^A(R)=+\infty$ or f-dim$_R(A)=+\infty$, then we are done.
Assume that dim$_e^A(R)<+\infty$ and f-dim$_R(A)<+\infty$. Let
$P_0\subset P_1\subset\cdots\subset P_n$ be a chain of distinct
prime ideals of $A$. Then the corresponding chain of contractions
$p_0\subset p_1\subset\cdots\subset p_r$ is composed of effective
prime ideals of $R$ with respect to $A$. Observe that the number of
the $P_i$'s lying over a fixed prime $p_j$ is inferieur than or
equal to the Krull dimension of the fibre ring $k_R(p_j)\otimes_RA$
plus one, that is, dim$_{p_j}(A)+1$, and dim$_{p_j}(A)\leq$
f-dim$_R(A)$. Further, as $p_0\subset p_1\subset\cdots\subset p_r$
is a chain composed of $1+r$ effective prime ideals of $R$ with
respect to $A$ and as $r\leq$ dim$_e^A(R)$, we get
$$\begin{array}{lll}n&\leq&
r(\mbox {f-dim}_R(A)+1)+\mbox {f-dim}_R(A)\\
&\leq&\mbox { f-dim}_R(A)+(\mbox {f-dim}_R(A)+1)\mbox {dim}_e^A(R).\end{array}$$ This yields the desired inequality.\\
2) Assume that $f_A$ satisfies GD. Let $n:=ht_e^A(p)$ and
$p_0\subset p_1\subset\cdots\subset p_n=p$ be a chain of distinct
effective primes of $R$ with respect to $A$. Fixing $P\in$
Spec$_{p}(A)$ and applying the Going-down property yields the
existence of a chain $P_0\subset P_1\subset\cdots\subset P_n=P$ such
that each $P_i\in$ Spec$_{p_i}(A)$ for $i=0,1,\cdots,n-1$. Then
$ht_e^A(p)=n\leq ht(P)$ for each $P\in$ Spec$_{p}(A)$. Let
$Q_0\subset Q_1\subset\cdots\subset Q_r$ be a chain of Spec$_p(A)$
such that $r:=$ dim$(k_R(p)\otimes_RA)=$ dim$_p(A)$. Therefore, as
by the first step $ht_e^A(p)=n\leq ht(Q_0)$, we get
$$ht_e^A(p)+\mbox
{dim}_p(A)\leq ht(Q_0)+r\leq ht(Q_r)\leq\mbox {dim}(A)$$for each
effective prime ideal $p$ of $R$. Hence, by (1), as
$ht_e^A(p)=ht(p)$, we get the desired inequalities completing the
proof.\end{proof}

Next, we list various applications and consequences of Theorem 2.9.
The first result gives a condition for coincidence of the Krull
dimension of $A$ and its fiber Krull dimension with respect to $R$.

\begin{cor} Let $A$ be an $R$-algebra. If dim$_e^{A}(R)=0$, then $$\mbox {dim}(A)=\mbox {f-dim}_R(A).$$
\end{cor}

We aim via the following corollaries to recover Seidenberg's
inequalities for polynomial rings. The next result might be termed
Seidenberg's inequalities for algebras over an arbitrary ring.

\begin{cor} Let $A$ be an $R$-algebra such that $f_A$
satisfies GD. Assume that f-dim$_R(A)=$ dim$_m(A)$ for each $m\in$
Max$_e^A(R)$. Then $$\mbox {f-dim}_R(A)+\mbox {dim}_e^A(R)\leq\mbox
{dim}(A)\leq \mbox {f-dim}_R(A)+(1+\mbox {f-dim}_R(A))\mbox
{dim}_e^{A}(R).$$\end{cor}

\begin{proof} Observe, by Theorem 2.9, that, in particular,
sup$\{ht_e^A(m)+$dim$_m(A):m\in$ Max$_e^A(R)\}\leq$ dim$(A)$. Then,
as f-dim$_R(A)=$ dim$_m(A)$ for each $m\in$ Max$_e^A(R)$,
sup$\{ht_e^A(m):m\in$ Max$_e^A(R)\}+$f-dim$_R(A)\leq$ dim$(A)$.
Therefore\\ dim$_e^A(R)+$f-dim$_R(A)\leq$ dim$(A)$ establishing the
desired inequalities.\end{proof}

Next, we recover Seidenberg's inequalities for polynomial rings.

\begin{cor} Let $R$ be a ring and let $X_1,X_2,\cdots,X_n$ be
indeterminates over $R$. Then $$n+\mbox {dim}(R)\leq\mbox
{dim}(R[X_1,X_2,\cdots,X_n])\leq n+(n+1)\mbox{dim}(R).$$\end{cor}

\begin{proof} Observe that the homomorphism $R\longrightarrow R[X_1,X_2,\cdots,X_n]$ satisfies GD and Spec$_e^{R[X_1,\cdots,X_n]}(R)=$
Spec$(R)$, thus dim$_e^{R[X_1,\cdots,X_n]}(R)=$ dim$(R)$. Also,
$$\begin{array}{lll}\mbox {dim}_p(R[X_1,X_2,\cdots,X_n])&=&\mbox {
dim}(k_R(p)\otimes_RR[X_1,X_2,\cdots,X_n])\\
&=&\mbox { dim}(k_R(p)[X_1,X_2,\cdots,X_n])=n\end{array}$$ for each
prime ideal $p$ of $R$. Then
$$\mbox {f-dim}_R(R[X_1,X_2,\cdots,X_n])=\mbox {
dim}_p(R[X_1,X_2,\cdots,X_n])=n$$ for each prime ideal $p$ of $R$.
Now, Corollary 2.11 completes the proof.\end{proof}

\section{Local and global Transcendence degree over an arbitrary
ring}

In this section we introduce the local and global transcendence
degree of algebras over an arbitrary ring.

Recall that if $k$ is a field, then it is customary to denote by
$$\begin{array}{lll}\mbox {t.d.}(A:k)&:=&\mbox { sup}\Big \{\mbox {t.d.}\Big (\displaystyle {\frac Ap}:k\Big ):p\in\mbox { Spec}(A)\Big \}\\
&=&\mbox { sup}\{\mbox {t.d.}(k_A(p):k):p\in\mbox {
Spec}(A)\}\end{array}$$
 the
transcendence degree of a $k$-algebra $A$ over $k$. This section
aims at giving a definition of the notion of the transcendence
degree of an $R$-algebra $A$ over $R$ in accordance with the field
case.

In the same spirit of Definition 2.7 and in order to prepare the
ground for the general case of tensor products over an arbitrary
ring, we next introduce a local notion of the transcendence degree
of an $R$-algebra $A$ over $R$ as well as the ``general"
transcendence degree of $A$ over $R$ which turns out to be in total
accordance with the well known notion of the transcendence degree
over a field.

\begin{defin} {\it Let $R$ be a ring and $A$ an $R$-algebra.}\\
1) {\it Let $p\in$ Spec$_e^A(R)$. We define the transcendence degree
at $p$
 of $A$ over $R$ to be the the transcendence degree of the fibre ring $k_R(p)\otimes_RA$ (over
 $p$) over the field $k_R(p)$, that
 is,}
 $$\mbox {t.d.}_p(A:R):=\mbox {t.d.}\Big
(k_R(p)\otimes_RA:k_R(p)\Big ).$$ 2) {\it We define the
transcendence degree of $A$ over $R$ to be the supremum of the
transcendence degrees of $A$ over $R$ at effective prime ideals $p$
of $R$ with respect to $A$, that is},$$\begin{array}{lll}\mbox
{t.d.}(A:R)&:=&\mbox { sup}
\{\mbox {t.d.}_p(A:R):p\in\mbox { Spec}_e^A(R)\}\\
&=&\mbox { sup}\Big \{\mbox {t.d.}\Big (k_R(p)\otimes_RA:k_R(p)\Big
):p\in\mbox {Spec}_e^A(R)\Big \}.\end{array}$$\end{defin}

We will next prove that the transcendence degree of an $R$-algebra
$A$ over $R$ depends on the endowing structure of algebra of $A$
over $R$, namely on the ring homomorphism $f_A$.

The following proposition shows that the transcendence degree over a
ring $R$ at an effective prime ideal of $R$ shares all known
properties of the transcendence degree over a field $k$. Notice
that, when $P$ is a prime ideal of an $R$-algebra $A$ and
$p:=f_A^{-1}(P)$, then $f_A$ induces an isomorphism between
$\displaystyle {\frac Rp}$ and a subring of $\displaystyle {\frac
AP}$ which means that $\displaystyle {\frac Rp}$ might be identified
with a subring of $\displaystyle {\frac AP}$.

\begin{prop} Let $A$ be an $R$-algebra.\\
1) If $P$ is a prime ideal of $A$ and $p:=f_A^{-1}(P)$, then

\hspace{2cm}t.d.$_p\Big (\displaystyle {\frac AP}:R\Big )=\mbox
{t.d.}\Big (\displaystyle {\frac
AP:\frac Rp}\Big )=$ t.d.$\Big (k_A(P):k_R(p)\Big ).$\\
2) Let $p\in$ Spec$_e^A(R)$. Then,
$$\mbox {t.d.}_p(A:R)=\mbox { sup}\Big \{\mbox {t.d.}_p\Big
(\displaystyle {\frac AP}:R\Big ):P\in\mbox { Spec}_p(A)\Big \}.$$
3) If P is a prime ideal of $A$ with $p:=f_A^{-1}(P)$, then
$$\begin{array}{lll}\mbox {t.d.}_p(A_P:R)&=&\mbox {sup}\Big \{\mbox
{t.d.}_p\Big (\displaystyle {\frac AQ:R}\Big ):Q\in\mbox {
Spec}_p(A)\mbox { such that
}Q\subseteq P\Big \}\\
&=&\mbox { sup}\Big \{\mbox { t.d.}\Big (k_A(Q):k_R(p)\Big
):Q\in\mbox { Spec}_p(A)\mbox { with } Q\subseteq P\Big
\}\end{array}$$
$\begin{array}{lll}\mbox {4) t.d.}(A:R)&=&\mbox {
sup}\Big \{\mbox {t.d.}_p\Big (\displaystyle {\frac AP:R}\Big
):p\in\mbox { Spec}(R)\mbox { and
}P\in\mbox { Spec}_p(A)\Big \}\\
&=&\mbox { sup}\Big \{\mbox {t.d.}\Big (\displaystyle {\frac
AP:\frac Rp}\Big ):p\in\mbox { Spec}(R)\mbox { and }P\in\mbox {
Spec}_p(A)\Big \}.\end{array}$\end{prop}

\begin{proof} Let $S_p=\displaystyle {\frac Rp}\setminus \{\overline 0\}$ for each prime ideal $p$ of $R$.\\
1) Let $P$ be a prime ideal of $A$ and let $p:=f_A^{-1}(P)$. Then,
$$\begin{array}{lll}\mbox {t.d.}_p\Big (\displaystyle {\frac
AP}:R\Big )&=&\mbox { t.d.}\Big (k_R(p)\otimes_R\displaystyle
{\frac AP}:k_R(p)\Big )\\
&=&\mbox {t.d.}\Big (S_p^{-1}\displaystyle {\frac
{A/P}{p(A/P)}}:k_R(p)\Big )\mbox { (Lemma 2.2(1))}\\
&=&\mbox {t.d.}\Big (S_p^{-1}\displaystyle {\frac AP}:k_R(p)\Big )\mbox { as }p\mbox { }\displaystyle {\frac AP}=f_A(p)\displaystyle {\frac AP}=\{\overline 0\}\\
&=&\mbox {t.d.}\Big (\displaystyle {\frac AP:\frac Rp}\Big
),\end{array}$$as desired.\\
2) Let $p\in$ Spec$_e^A(R)$. Then, by (1) and Lemma 2.2((2) and
(5)),
$$\begin{array}{lll} \mbox {t.d.}_p(A:R)&:=&\mbox {
t.d.}\Big (k_R(p)\otimes_RA:k_R(p)\Big )\\
&=&\mbox { sup}\Big \{\displaystyle {\mbox {t.d.}\Big (\frac
{k_R(p)\otimes_RA}{k_R(p)\otimes_RP}}:k_R(p)\Big ): P\in\mbox {
Spec}_p(A)\Big \}\\
&=&\mbox { sup}\Big \{\displaystyle {\mbox {t.d.}\Big (
k_R(p)\otimes_R\frac AP}:k_R(p)\Big ): P\in\mbox {
Spec}_p(A)\Big \}\\
&=&\mbox { sup}\Big \{\mbox {t.d.}_p\Big (\displaystyle {\frac
AP}:R\Big ):P\in\mbox { Spec}_p(A)\Big \}\\
&=&\mbox {sup}\Big \{\mbox { t.d.}\Big (\displaystyle {\frac
AP:\frac Rp}\Big ):P\in\mbox { Spec}_p(A)\Big \}.\end{array}$$
\noindent 3) Let $P$ be a prime
ideal of $A$ and $p:=f_A^{-1}(P)$. Then, by (1) and (2),\\
$\begin{array}{lll} \mbox {t.d.}_p(A_P:R)&=&\mbox { sup}\Big \{\mbox
{t.d.}\Big (\displaystyle {\frac {A_P}{Q_P}:\frac Rp}\Big
):Q\in\mbox {
Spec}_p(A)\mbox { such that }Q\subseteq P\Big \}\\
&=&\mbox { sup}\Big \{\mbox {t.d.}\Big (\displaystyle {\frac
{A}{Q}:\frac Rp}\Big ):Q\in\mbox {
Spec}_p(A)\mbox { such that }Q\subseteq P\Big \}\\
&=&\mbox { sup}\Big \{\mbox {t.d.}_p\Big (\displaystyle {\frac
{A}{Q}:R}\Big ):Q\in\mbox {
Spec}_p(A)\mbox { such that }Q\subseteq P\Big \}.\end{array}$\\
4) It follows easily from (1) and (2) completing the
proof.\end{proof}

It is notable that the introduced notion of transcendence degree of
$A$ over $R$ depends on the structure of $R$-algebra over $A$,
namely, on the ring
homomorphism $f_A$, as shown by the following simple example.\\

\begin{exa} Let $k$ be a field and $X$ an indeterminate over $k$. Let $R=k[X]\times k$ and let
$A=k(X)$. Consider the following ring homomorphisms
$f_1,f_2:R\longrightarrow A$ such that $f_1(g(X),\alpha)=g(X)$ and
$f_2(g(X),\alpha)=\alpha$. These two homomorphisms define two
different $R$-algebra structures over $A$. Moreover, observe that
$f_1^{-1}((0))=(0)\times k$ and $f_2^{-1}((0))=k[X]\times (0)$. Then
Spec$_e^{(A,f_1)}(R)=\{(0)\times k\}$ and
Spec$_e^{(A,f_2)}(R)=\{k[X]\times (0)\}$. Hence, by Proposition
3.2(4),
$$\begin{array}{lll}\mbox {t.d.}(A:_{f_1}R)&=&\mbox {t.d.}\Big (k(X):\displaystyle {\frac {k[X]\times k}{(0)\times k}\Big )}\\
&=&\mbox {t.d.}(k(X):k[X])=0\end{array}$$ while
$$\begin{array}{lll}\mbox {t.d.}(A:_{f_2}R)&=&\mbox {t.d.}\Big (k(X):\displaystyle {\frac {k[X]\times k}{k[X]\times (0)}\Big )}\\
&=&\mbox {t.d.}(k(X):k)=1.\end{array}$$
\end{exa}

\section{Effective spectrum with respect to tensor products}

Let $R$ be a ring and let $A$ and $B$ be $R$-algebras. First, it is
worth to note that the tensor product $A\otimes_RB$ over $R$ might
be trivial even if $A$ and $B$ are not so. Of course, the
interesting case is when $A\otimes_RB\neq \{0\}$ which makes it
legitimate to introduce the notion of a triplet of rings $(R,A,B)$
consisting of a given ring $R$ and two $R$-algebras $A$ and $B$ such
that $A\otimes_RB\neq \{0\}$.

Let $A$ and $B$ be $R$-algebras. We denote by
$\mu_A:A\longrightarrow A\otimes_RB$ and $\mu_B:B\longrightarrow
A\otimes_RB$ the canonical algebra homomorphisms over $A$ and $B$,
respectively, such that $\mu_A(a)=a\otimes_R1$ and
$\mu_B(b)=1\otimes_Rb$ for each $a\in A$ and each $b\in B$. Observe
that the following diagram (D) is commutative:

$$\begin{array}{ccccc}
&&A&&\\
&\stackrel {f_A}\nearrow &&\stackrel {\mu_A}\searrow&\\
R&&\stackrel {f_{A\otimes_RB}}\longrightarrow&&A\otimes_RB\\
&\stackrel {f_B}\searrow&&\stackrel {\mu_B}\nearrow&\\
&&B&&
\end{array}$$

\noindent From this section onward, given ideals $I,J,H$ of $A$, $B$
and $A\otimes_RB$, respectively, we adopt the following notation for
easiness: $I\cap R:=f_A^{-1}(I)$, $J\cap R:=f_B^{-1}(J)$ and $H\cap
A:=\mu_A^{-1}(H)$, $H\cap B:=\mu_B^{-1}(H)$.\\

We begin by recording the following isomorphisms related to the
fiber rings of the tensor products over an arbitrary ring
$R$.\\

\begin{lem} {\it Let $R$ be a ring. Let $A$ and $B$
be algebras over $R$. Then}
$$k_R(p)\otimes_R(A\otimes_RB)\cong (k_R(p)\otimes_RA)\otimes_{k_R(p)}(k_R(p)\otimes_RB)\cong S_p^{-1}\Big (\frac A{pA}\Big )
\otimes_{k_R(p)} S_p^{-1}\Big (\frac B{pB}\Big )$$ where
$S_p:=\displaystyle {\frac Rp}\setminus \{0\}$.\end{lem}

\begin{proof} It is direct by Lemma 2.2(1).\end{proof}

\noindent {\it Remark.} {\it Let $R$ be a ring and $A,B$ be two
$R$-algebras. Let $I$ be a prime ideal of $A$. Notice that, by
considering the ring homomorphism $\mu_A:A\longrightarrow
A\otimes_RB$, $k_A(I)\otimes_RB\cong k_A(I)\otimes_A(A\otimes_RB)$
stands for the fibre ring of $A\otimes_RB$ over $I$. Thus
f-dim$_A(A\otimes_RB)=$ sup$\{$dim$(k_A(I)\otimes_RB):I\in$
Spec$_e^{A\otimes_RB}(A)\}$.}\\

The next theorem examines the effective spectrum of a ring with
respect to tensor products.

\begin{thm} Let $R$ be a ring. Let $A$ and $B$ be two $R$-algebras. Let $p\in$ Spec$(R)$, $I\in$ Spec$(A)$ and $J\in$
Spec$(B)$. Then\\
1) Spec$_e^{A\otimes_RB}(R)=$ Spec$_e^A(R)\cap$ Spec$_e^B(R)$.\\
2) There exists a prime ideal $P$ of $A\otimes_RB$ such that $P\cap
A=I$ and $P\cap B=J$ if and only if $I\cap R=J\cap R$.\\
3) $$\begin{array}{lll} \mbox {Spec}_e^{A\otimes_RB}(A)&=&\{I\in\mbox {Spec}(A):I\cap R\in\mbox {Spec}_e^{B}(R)\}\\
&=&\{I\in\mbox {Spec}(A):\exists J\in\mbox {Spec}(B)\mbox { such
that }I\cap R=J\cap R\}.\end{array}$$
\end{thm}

\begin{proof} 1) Let $p\in$ Spec$_e^{A\otimes_RB}(R)$. Then there exists a prime ideal $P$ of
$A\otimes_RB$ such that $p=P\cap R$. Let $I=P\cap A$ and $J=P\cap
B$. Then, by the above commutative diagram (D), $I$ and $J$ are
prime ideals of $A$ and $B$, respectively, such that $I\cap R=J\cap
R=P\cap R=p$, that is, $p\in$ Spec$_e^A(R)\cap$ Spec$_e^B(R)$.
Conversely, assume that $p\in$ Spec$_e^A(R)\cap$ Spec$_e^B(R)$.
Hence, by Lemma 4.1, the fibre ring
$$k_R(p)\otimes_R(A\otimes_RB)\cong(k_R(p)\otimes_RA)\otimes_{k_R(p)}(k_R(p)\otimes_RB)\neq
\{0\}$$ as the fibre rings $k_R(p)\otimes_RA$ and $k_R(p)\otimes_RB$
are not trivial. It follows that $p\in$ Spec$_e^{A\otimes_RB}(R)$,
as
desired.\\
2) See \cite[Corollaire 3.2.7.1.(i)]{EGA1}.\\
3) First, let $I\in$ Spec$(A)$ such that $I\cap R\in$ Spec$_e^B(R)$.
Then, there exists $J\in$ Spec$(B)$ such that $I\cap R=J\cap R$ so
that by (2), there exists $P\in$ Spec$(A\otimes_RB)$ such that
$P\cap A=I$ (and $P\cap B=J$). Thus $I\in$
Spec$_e^{A\otimes_RB}(A)$. Conversely, let $I\in$
Spec$_e^{A\otimes_RB}(A)$. Then there exists $P\in$
Spec$(A\otimes_RB)$ such that $P\cap A=I$, so that, using the above
commutative diagram (D), $I\cap R=P\cap R=(P\cap B)\cap R\in$
Spec$_e^B(R)$ completing the proof.
\end{proof}

The following corollary totally characterizes when two algebras $A$
and $B$ over a ring $R$ constitute a triplet $(R,A,B)$ of rings.

\begin{cor} Let $R$ be a ring and $A,B$ be two $R$-algebras. Then the following assertions are equivalent:

1) $(R,A,B)$ is a triplet of rings;

2) Spec$_e^A(R)\cap$ Spec$_e^B(R)\neq\emptyset$;

3) There exists a prime ideal $I$ of $A$ and a prime ideal $J$ of
$B$ such that $I\cap R=J\cap R$.\end{cor}

\begin{proof}
1) $\Leftrightarrow$ 2) It suffices to observe that, by Proposition
2.4(1) and Theorem 4.2,
$$f_{A\otimes_RB}^{-1}(\mbox {Spec}(A\otimes_RB))=\mbox {Spec}_e^{A\otimes_RB}(R)=\mbox { Spec}_e^A(R)\cap\mbox {
Spec}_e^B(R).$$ \noindent 2) $\Leftrightarrow$ 3) It is
direct.\\
\end{proof}

It is easy to provide examples of nontrivial algebras over a ring
$R$ such that $A\otimes_RB=\{0\}$ is trivial. But Corollary 4.3
characterizes when this tensor product $A\otimes_RB$ is trivial by
checking connections between the spectrum of the three components of
this construction, namely Spec$(R)$, Spec$(A)$ and Spec$(B)$. For
instance, given a ring $R$ and two distinct prime ideals $p$ and $q$
of $R$, applying Corollary 4.3, note that
$k_R(p)\otimes_Rk_R(q)=\{\overline 0\}$ since
Spec$_e^{k_R(p)}(R)=\{p\}$ while Spec$_e^{k_R(q)}(R)=\{q\}$ and thus
Spec$_e^{k_R(p)}(R)\cap$
Spec$_e^{k_R(q)}(R)=\emptyset$.\\

\begin{cor} Let $(R,A,B)$ be a triplet of rings. Then $$\mbox {dim}_e^{A\otimes_RB}(R)\leq\mbox {min}(\mbox
{dim}_e^A(R),\mbox {dim}_e^B(R)).$$ In particular, if either
dim$_e^A(R)=0$ or dim$_e^B(R)=0$, then
dim$_e^{A\otimes_RB}(R)=0$.\end{cor}

\begin{proof} It is straightforward from Theorem 4.2 as
Spec$_e^{A\otimes_RB}(R)=$ Spec$_e^A(R)\cap$
Spec$_e^B(R)$.\end{proof}

The next proposition allows us to give lower and upper bounds of the
Krull dimension of tensor products of algebras over a ring $R$ in
terms of the Krull dimension of its fibre rings and the effective
Krull dimension of its components.

\begin{prop} Let $(R,A,B)$ be a triplet of rings. Then\\
1) f-dim$_R(A\otimes_RB)\leq$ dim$(A\otimes_RB)\leq$
f-dim$_R(A\otimes_RB)+(1+$f-dim$_R(A\otimes_RB))$dim$_e^{A\otimes_RB}(R)$.\\
2) f-dim$_A(A\otimes_RB)\leq$ dim$(A\otimes_RB)\leq$
f-dim$_A(A\otimes_RB)+(1+$f-dim$_A(A\otimes_RB))$dim$_e^{A\otimes_RB}(A)$.\end{prop}

\begin{proof} It follows from Theorem 2.9(1).\end{proof}

In light of Proposition 4.5, when the effective Krull dimension of a
component of a tensor product with respect to this construction is
zero, the Krull dimension of the tensor product turns out to be its
own fibre Krull dimension. This result will allow us to explicit the
Krull dimension of the tensor products involving the ahead
introduced notion of fibred
AF-rings.\\

\begin{cor} Let $(R,A,B)$ be a triplet of rings.\\
1) If dim$_e^{A\otimes_RB}(R)=0$, then $$\mbox
{dim}(A\otimes_RB)=\mbox {f-dim}_R(A\otimes_RB).$$ 2) If
dim$_e^{A\otimes_RB}(A)=0$, then
$$\mbox {dim}(A\otimes_RB)=\mbox {f-dim}_A(A\otimes_RB).$$\end{cor}

\begin{proof} It follows easily from Proposition 4.5.
\end{proof}

\section{Fibred AF-rings}

This section introduces and studies the notion of fibred AF-rings
over an arbitrary ring $R$.  This new concept extends that of
AF-ring over a field introduced by A. Wadsworth in \cite{W}.\\

\begin{defin} {\it Let $R$ be a ring.}\\
1) {\it Let $A$ be an $R$-algebra and $p\in$ Spec$_e^A(R)$. $A$ is
said to be a fibred AF-ring at $p$ if its fiber ring
$k_R(p)\otimes_RA$ is an
AF-ring over $k_R(p)$.}\\
2) {\it An $R$-algebra $A$ is said to be a fibred AF-ring over $R$
if it is a fibred AF-ring at each effective prime ideal $p$ of $R$
with respect to $A$,
 that is, if each nontrivial fibre ring $k_R(p)\otimes_RA$ is an AF-ring over $k_R(p)$.}\\
3) {\it Let $(R,A,B)$ be a triplet of rings.}

a) {\it If $p\in$ Spec$_e^A(R)\cap$ Spec$_e^B(R)$ and $A$ is a
fibred AF-ring at $p$, then $A$ is said to be a $B$-fibred AF-ring
at $p$.}

b) {\it $A$ is said to be a $B$-fibred AF-ring over $R$ if $A$ is a
$B$-fibred AF-ring at each common effective prime ideal $p$ of $R$
with respect to $A$ and $B$}.\end{defin}

\noindent {\it Remark.} 1) From this definition, it is clear that
the notion of a fibred AF-ring extends that of an AF-ring introduced
by Wadsworth as any algebra $A$ over a field $k$ possesses only one fiber ring over $k$ which is $A$ itself.\\
2) Note that the notion of $B$-fibred AF-ring is inherent to the
considered triplet of rings $(R,A,B)$ and it is easy to see that the
following assertions are equivalent:

a) $A$ is a fibred AF-ring over $R$;

c) $A$ is a $B$-fibred AF-ring over $R$ for any $R$-algebra $B$ such
that $(R,A,B)$ is a triplet;

c) $A$ is an $R$-fibred AF-ring over $R$.\\

Let $k$ be a field and $A$ be a $k$-algebra. Recall that, for each
prime ideal $P$ of $A$, $ht(P)+$t.d.$\Big (\displaystyle {\frac
AP}:k\Big )\leq$ t.d.$(A_P:k)$ and $A$ is said to be an AF-ring if
this inequality turns out to be an equality, that is,
$ht(P)+$t.d.$\Big (\displaystyle {\frac AP}:k\Big )=$ t.d.$(A_P:k)$
for each prime ideal $P$ of $A$. Next, we show that these properties
translate into local data for algebras over an arbitrary ring $R$.

\begin{prop} {\it Let $R$ be a ring and let $p$ be a prime ideal of
$R$. Let $A$ be an $R$-algebra.\\
1) If $P$ is a prime ideal of $A$ such that $P\cap R=p$, then
$$ht_p(P)+\mbox {t.d.}_p\Big (\displaystyle {\frac AP}:R\Big
)\leq\mbox { t.d.}_p(A_P:R).$$ 2) Let $B$ be an $R$-algebra such
that $(R,A,B)$ is a triplet of rings. Assume that $p\in$
Spec$_e^A(R)$ (resp., $p\in$ Spec$_e^A(R)\cap$ Spec$_e^B(R)$). Then
$A$ is a fibred AF-ring (resp., $B$-fibred AF-ring) at $p$ if and
only if $$ht_p(P)+\mbox {t.d.}_p\Big (\frac {A}{P}:R\Big )=\mbox {
t.d.}_p(A_P:R)$$ for each prime ideal $P$ of $A$ such that $p=P\cap
R$.}\end{prop}

\begin{proof} 1) Let $P$ be a prime ideal of $A$ such that
$P\cap R=p$. Then, by Lemma 2.2,
$$\begin{array}{lll}ht_p(P)+\mbox {t.d.}_p\Big (\displaystyle {\frac
AP}:R\Big )&=&ht(k_R(p)\otimes_RP)+\mbox
{t.d.}\Big (k_R(p)\otimes_R\displaystyle {\frac AP}:k_R(p)\Big )\\
&=&ht(k_R(p)\otimes_RP)+\mbox
{t.d.}\Big (\displaystyle {\frac {k_R(p)\otimes_RA}{k_R(p)\otimes_RP}}:k_R(p)\Big )\\
&\leq&\mbox {t.d.}\Big
((k_R(p)\otimes_RA)_{k_R(p)\otimes_RP}:k_R(p)\Big )\\
&=&\mbox { t.d.}\Big (k_R(p)\otimes_RA_P:k_R(p)\Big
)\\
&=&\mbox { t.d.}_p(A_P:R),\end{array}$$as desired.\\
2) It is direct from Definition 5.1 taking into account the
following equalities which figure in Lemma 2.2:  t.d.$_p\Big
(\displaystyle {\frac AP}:R\Big )=$ t.d.$\Big (\displaystyle {\frac
{k_R(p)\otimes_RA}{k_R(p)\otimes_RP}:k_R(p)\Big )}$ and $\mbox {
t.d.}_p(A_P:R)=\mbox {t.d.}\Big
(k_R(p)\otimes_RA)_{k_R(p)\otimes_RP}:k_R(p)\Big )$ for each prime
ideal $P$ of $A$ with $p:=P\cap R$.\end{proof}

The following result exhibits various classes of fibred AF-rings.\\

\begin{prop} 1) If $k$ is a field,
then the fibred AF-rings over $k$ are exactly the
AF-rings over $k$.\\
2) Any finitely generated $R$-algebra $R[x_1,x_2,...,x_n]$ is a
fibred AF-ring over
$R$.\\
3) Let $(R,A,B)$ be a triplet of rings. If dim$_e^{A\otimes_RB}(A)=0$, then $A$ is a $B$-fibred AF-ring over $R$.\\
4) Any zero-dimensional ring $A$ which is an $R$-algebra is a fibred
AF-ring over $R$.\\
\end{prop}

\begin{proof} 1) It is straightforward.\\
2) Let $A=R[X_1,X_2,\cdots,X_n]$ be a polynomial ring in
$n$-variables over $R$. Let $p\in$ Spec$(R)$. Then
$k_R(p)\otimes_RA\cong k_R(p)[X_1,X_2,...,X_n]$ is clearly an
AF-ring over $k_R(p)$ \cite[Corollary 3.2]{W}. Hence $A$ is a fibred
AF-ring over $R$. Now, let $A=R[x_1,x_2,\cdots,x_n]$ be any finitely
generated $R$-algebra and let $p\in$ Spec$(R)$. Then $A\cong
\displaystyle {\frac {R[X_1,X_2,\cdots,X_n]}I}$ for some ideal $I$
of $R[X_1,X_2,\cdots,X_n]$ and thus, by Lemma 2.2,
$$\begin{array}{lll} k_R(p)\otimes_RA&\cong&
k_R(p)\otimes_R\displaystyle {\frac {R[X_1,X_2,\cdots,X_n]}I}\\
&&\\
&\cong&(k_R(p)\otimes_RR[X_1,X_2,\cdots,X_n])\otimes_{R[X_1,X_2,\cdots,X_n]}\displaystyle
{\frac {R[X_1,X_2,\cdots,X_n]}I}\\
&&\\
&\cong& \displaystyle {\frac
{k_R(p)[X_1,X_2,\cdots,X_n]}{Ik_R(p)[X_1,X_2,\cdots,X_n]}}:=k_R(p)[y_1,y_2,\cdots,y_n]\mbox
{, where }y_i:=\overline {X_i}\\
&&\mbox { for each }i=1,2,\cdots,n,\end{array}$$ is a finitely
generated $k_R(p)$-algebra which is an AF-ring over
$k_R(p)$ by \cite[page 395]{W}. Hence $A$ is a fibred AF-ring over $R$ proving (2).\\
3) Let $(R,A,B)$ be a triplet of rings such that
dim$_e^{A\otimes_RB}(A)=0$. Let\\
$p\in$ Spec$_e^A(R)\cap$ Spec$_e^B(R)=$ Spec$_e^{A\otimes_RB}(R)$.
Let $P$ be a prime ideal of $A$ such that $P\cap R=p$. Then, by
Proposition 3.2,
$$\begin{array}{lll}\displaystyle {\mbox { t.d.}_p
(A_P:R)}&=&\displaystyle {\mbox { sup}\Big \{\mbox { t.d.}\Big
(\frac {A}{Q}:\frac Rp}\Big ):Q\in\mbox { Spec}_p(A)\mbox { with }
Q\subseteq P\Big \}.\end{array}$$ Let $Q\subseteq P$ be a prime
ideal of $A$ with $Q\cap R=p$. Then, as $p\in$
Spec$_e^{A\otimes_RB}(R)$, by Theorem 4.2(3), $Q\in$
Spec$_e^{A\otimes_RB}(A)$. Now, since dim$_e^{A\otimes_RB}(A)=0$, we
get $Q=P$. Therefore, by Proposition 3.2,
$$\mbox { t.d.}_p(A_P:R)
=\displaystyle {\mbox { t.d.}\Big (\frac {A}{P}:\frac Rp\Big
)}=\mbox { t.d.}_p\Big (\displaystyle {\frac AP}:R\Big ).$$ Also, as
dim$_e^{A\otimes_RB}(A)=0$, ht$\Big (\displaystyle {\frac P{pA}}\Big
)=0$ since any prime ideal $Q$ such that $pA\subseteq Q\subseteq P$
is an effective prime ideal of $A$ with respect to $A\otimes_RB$
(see Theorem 4.2(3)). Then, by Corollary 2.8 , $ht_p(P)=ht\Big
(\displaystyle {\frac P{pA}}\Big )=0$. It follows that
$$ht_p(P)+ \displaystyle {\mbox { t.d.}_p\Big (\frac {A}{P}:R}\Big
)=\mbox {
 t.d.}_p(A_P:R).$$ Then, by Proposition 5.2, $A$
is a $B$-fibred AF-ring over $R$, as desired.\\
4) Note that if dim$(A)=0$, then, in particular,
dim$_e^{A\otimes_RB}(A)=0$ for any triplet $(R,A,B)$ of rings.
Hence, by (3), $A$ is a $B$-fibred AF-ring over $R$ for any
$R$-algebra $B$ such that $(R,A,B)$ is a triplet of rings. In
particular for $B=R$, $A$ is an $R$-fibred AF-ring over $R$ which
means that $A$ is a fibred AF-ring over $R$, as desired.
\end{proof}

We next establish the stability of the fibred AF-ring notion under
various type of constructions.

\begin{prop} Let $R$ be a ring and let $A$ be an $R$-algebra. Let $p$ be a prime ideal of $R$.\\
1) If $A$ is a fibred AF-ring at $p$ and $S$ is a multiplicative
subset of $A$ such that $p\in$ Spec$_e^{S^{-1}A}(R)$, then
the localization $S^{-1}A$ is a fibred AF-ring at $p$.\\
2) Let $A_1,A_2,...,A_n$ be fibred AF-rings at $p$. Then
$A_1\otimes_RA_2\otimes_R\cdots\otimes_RA_n$ is a fibred AF-ring
at $p$.\\
3) If $A$ is a fibred AF-ring at $p$, then the polynomial ring
$A[X_1,X_2,...,X_n]$ is a fibred AF-ring at $p$.
\end{prop}

\begin{proof} 1) Let $A$ be a fibred AF-ring at $p$ and $S$ be a
multiplicative subset of $A$ such that $k_R(p)\otimes_RS^{-1}A\neq
\{0\}$. Then, by Proposition 2.4(5), there exists $P\in$ Spec$(A)$
such that $P\cap S=\emptyset$ and $p=P\cap R$. Let $\overline S$ be
the image of $S$ via the homomorphism $A\longrightarrow
\displaystyle {\frac A{pA}}$. Observe that $pA\cap S=\emptyset$ and
$\displaystyle {\frac P{pA}}\cap \overline S=\emptyset$. Then, by
Lemma 2.2,
$$\begin{array}{lll}ht_p(S^{-1}P)+\mbox {t.d.}_p\Big (\displaystyle {\frac {S^{-1}A}{S^{-1}P}}:R\Big )&=&ht\Big ( \displaystyle {\frac
{S^{-1}P}{pS^{-1}A}\Big )+\mbox {t.d.}\Big (\frac
{S^{-1}A}{S^{-1}P}:\frac Rp\Big )}\\
&=&\displaystyle {ht\Big
(\overline S^{-1}\frac {P}{pA}\Big )+\mbox {t.d.}\Big (\frac
{A}{P}:\frac Rp\Big )}\\
&=&\displaystyle {ht\Big (\frac {P}{pA}\Big )+\mbox {t.d.}\Big
(\frac {A}{P}:\frac Rp\Big )}\\
&=&ht_p(P)+\mbox {t.d.}_p\Big (\displaystyle {\frac AP}:R\Big )\\
&=&\mbox {t.d.}_p(A_P:R)\end{array}$$as $A$ is a fibred AF-ring at
$p$. Moreover, note that t.d.$_p((S^{-1}A)_{S^{-1}P}:R)=$
t.d.$_p(A_P:R)$. It follows that $$ht_p(S^{-1}P)+\mbox {t.d.}_p\Big
(\displaystyle {\frac {S^{-1}A}{S^{-1}P}}:R\Big )=\mbox {
t.d.}_p((S^{-1}A)_{S^{-1}P}:R).$$ Hence $S^{-1}A$ is a fibred AF-ring at $p$.\\
2) Let $p$ be a prime ideal of $R$. Then
$$\begin{array}{lll}k_R(p)\otimes_R(A_1\otimes_R\cdots\otimes_RA_n)&\cong&(k_R(p)\otimes_RA_1)\otimes_{k_R(p)}(k_R(p)\otimes_R(A_2\otimes_R\cdots\otimes_RA_n))\\
&\cong&(k_R(p)\otimes_RA_1)\otimes_{k_R(p)}\cdots\otimes_{k_R(p)}(k_R(p)\otimes_RA_n).\end{array}$$
 First, as each $k_R(p)\otimes_RA_i\neq
\{0\}$,
$k_R(p)\otimes_R(A_1\otimes_RA_2\otimes_R\cdots\otimes_RA_n)\neq
\{0\}$. Hence, since each $k_R(p)\otimes_RA_i$ is an AF-ring over
$k_R(p)$, we get, by \cite[Proposition 3.1]{W},
$(k_R(p)\otimes_RA_1)\otimes_{k_R(p)}\cdots\otimes_{k_R(p)}(k_R(p)\otimes_RA_n)$
is an AF-ring over $k_R(p)$ so that
$k_R(p)\otimes_R(A_1\otimes_R\cdots\otimes_RA_n)$ is an AF-ring over
$k_R(p)$. It follows that $A_1\otimes_R\cdots\otimes_RA_n$ is a
fibred AF-ring at $p$.\\
3) It follows easily from (2) as $A[X_1,X_2,...,X_n]\cong
R[X_1,X_2,...,X_n]\otimes_RA$ and, by Proposition 5.3(2),
$R[X_1,X_2,...,X_n]$ is a fibred AF-ring at $p$.\end{proof}

\begin{cor} Let $R$ be a ring.\\
1) If $A$ is a fibred AF-ring over $R$ and $S$ is a multiplicative
subset of $A$, then
the localization $S^{-1}A$ is a fibred AF-ring over $R$.\\
2) Let $A_1,A_2,...,A_n$ be fibred AF-rings over $R$. Then
$A_1\otimes_RA_2\otimes_R\cdots\otimes_RA_n$ is a fibred AF-ring
over $R$.\\
3) If $A$ is a fibred AF-ring over $R$, then the polynomial ring
$A[X_1,X_2,...,X_n]$ is a fibred AF-ring over $R$.
\end{cor}

The following two corollaries give the $B$-fibred AF-ring versions
of the above Proposition 5.4 and Corollary 5.5. Their proofs are
straightforward.

\begin{cor} Let $(R,A,B)$ be a triplet of rings. Let $p$ be a prime ideal of $R$.\\
1) If $A$ is a $B$-fibred AF-ring at $p$ and $S$ is a multiplicative
subset of $A$ such that $p\in$ Spec$_e^{S^{-1}A}(R)$, then
the localization $S^{-1}A$ is a $B$-fibred AF-ring at $p$.\\
2) Let $A_1,A_2,...,A_n$ be $B$-fibred AF-rings at $p$. Then
$A_1\otimes_RA_2\otimes_R\cdots\otimes_RA_n$ is a
$B$-fibred AF-ring at $p$.\\
3) If $A$ is a $B$-fibred AF-ring at $p$, then the polynomial ring
$A[X_1,X_2,...,X_n]$ is a $B$-fibred AF-ring at $p$.
\end{cor}

\begin{cor} Let $(R,A,B)$ be a triplet of rings.\\
1) If $A$ is a $B$-fibred AF-ring over $R$ and $S$ is a
multiplicative subset of $A$, then
the localization $S^{-1}A$ is a $B$-fibred AF-ring over $R$.\\
2) Let $A_1,A_2,...,A_n$ be $B$-fibred AF-rings over $R$. Then
$A_1\otimes_RA_2\otimes_R\cdots\otimes_RA_n$ is a
$B$-fibred AF-ring over $R$.\\
3) If $A$ is a $B$-fibred AF-ring over $R$, then the polynomial ring
$A[X_1,X_2,...,X_n]$ is a $B$-fibred AF-ring over $R$.
\end{cor}

\section{Krull dimension of tensor products involving fibred
AF-rings}

The goal of this section is to discuss and compute the Krull
dimension of the tensor product of algebras over $R$ involving
fibred AF-rings in various settings.\\

The following theorem allows to compute the Krull dimension of all
fibre rings of the tensor product of algebras $A$ and $B$ over a
ring $R$ in the case when $A$ is a fibred AF-ring over $R$. This
result translates Wadsworth theorem \cite[Theorem 3.7]{W} into the
general setting of tensor products over an arbitrary ring $R$. We
give the next more general version of a triplet $(R,A,B)$ of rings
such that $A$ is a $B$-fibred AF-ring over an effective prime ideal
$p$ of $R$.\\

\noindent {\it Notation.} 1) Let $A$ be a ring and $P$ be a prime
ideal of $A$. Let $n\geq 1$ be a positive integer. Then, for
easiness of notation, we denote by $A[n]$ the polynomial ring in $n$
indeterminates $A[X_1,X_2,\cdots,n]$ and by $P[n]$ the extended
prime ideal $P[X_1,X_2,\cdots,X_n]$ of $A[X_1,X_2,\cdots,X_n]$.\\
2) Let $A$ be an algebra over a field $k$. Let $0\leq d\leq s$ be
positive integers. Then, in \cite{W}, Wadsworth adopted the
following notation: $$D(s,d,A):=\mbox { sup}\Big \{ht(P[s])+\mbox
{min}\Big (s,d+\mbox {t.d.}\Big (\displaystyle {\frac AP}:k\Big
)\Big ):P\in\mbox { Spec}(A)\Big \}.$$ 3) Let $A$ be an algebra over
a ring $R$ and $p$ be a prime ideal of $R$. Let $0\leq d\leq s$ be
positive integers. Then, we adopt the following notation for a local
invariant of the above $D(s,d,A)$:
$$D_p(s,d,A):=D(s,d,k_R(p)\otimes_RA).$$\\

We begin by expliciting the local invariant $D_p(s,d,A)$ in terms of
the local invariants of the height and transcendence degree.

\begin{lem} Let $R$ be a ring. Let $A$ be an algebra over $R$ and $p$ be a prime ideal of $R$. Let $0\leq d\leq s$ be
positive integers. Then $$D_p(s,d,A)=\mbox { sup}\Big
\{ht_p(P[s])+\mbox {min}\Big (s,d+\mbox {t.d.}_p\Big (\displaystyle
{\frac AP}:R\Big )\Big ):P\in\mbox { Spec}_p(A)\Big \}.$$\end{lem}

It is worth noting that if $A$ and $B$ are algebras over a ring $R$
and $X_1,X_2,\cdots,X_n$ are indeterminates, then
$$\begin{array}{lll}(A\otimes_RB)[X_1,X_2,\cdots,X_n]&\cong&
A\otimes_RB\otimes_RR[X_1,X_2,\cdots,X_n]\\
&\cong& A[X_1,X_2,\cdots,X_n]\otimes_RB\\
&\cong& A\otimes_RB[X_1,X_2,\cdots,X_n].\end{array}$$

\begin{proof} Observe that, using Lemma 2.2 and Corollary 2.8,
$$\begin{array}{lll}D_p(s,d,A)&=&D(s,d,k_R(p)\otimes_RA)\\
&=&\mbox { sup}\Big \{ht(k_R(p)\otimes_RP)[s]+\mbox {min}\Big
(s,d+\mbox {t.d.}\Big (\displaystyle {\frac
{k_R(p)\otimes_RA}{k_R(p)\otimes_RP}}:k_R(p)\Big )\Big ):\\
&&P\in\mbox {
Spec}_p(A)\Big \}\\
&=&\mbox { sup}\Big \{ht(k_R(p)\otimes_R(P[s]))+\mbox {min}\Big
(s,d+\mbox {t.d.}\Big (k_R(p)\otimes_R\displaystyle {\frac
{A}{P}}:k_R(p)\Big )\Big ):\\
&&P\in\mbox { Spec}_p(A)\Big \}\\
&=&\mbox { sup}\Big \{ht_p(P[s])+\mbox {min}\Big (s,d+\mbox
{t.d.}_p\Big (\displaystyle {\frac {A}{P}}:R\Big )\Big ):P\in\mbox {
Spec}_p(A)\Big \},\end{array}$$ as desired.
\end{proof}

Recall that Wadsworth proved in \cite{W} that, given a field $k$, if
$A$ is an AF-domain and $B$ is any $k$-algebra, then
$$\mbox {dim}(A\otimes_kB)=D(\mbox {t.d.}(A),\mbox {dim}(A),B) \mbox {\cite[Theorem 3.7]{W}}.$$ We generalized this
result in \cite{BGK1} to AF-rings by proving that if $A$ is an
AF-ring and $B$ is any $k$-algebra, then,
$$\mbox {dim}(A\otimes_kB)=\mbox { sup}\Big \{D\Big (\mbox
{t.d.}(A_P:k),\mbox { dim}(A_P),B\Big ):P\in\mbox { Spec}(A)\Big
\}\mbox { \cite[Theorem 1.4]{BGK1}}.$$

Our first main result gives a new version of the above-cited
Wadsworth theorem in the general setting of tensor products of
algebras over an arbitrary ring $R$.

\begin{thm}  Let $(R,A,B)$ be a triplet of rings and let $p\in$ Spec$_e^A(R)\cap$ Spec$_e^B(R)$. Assume that $A$ is a $B$-fibred AF-ring
at $p$. Then
$$\mbox {dim}_p(A\otimes_RB)=\mbox { sup}\Big \{D_p\Big (\mbox
{t.d.}_p(A_I:R),ht_p(I),B\Big ):I\in\mbox { Spec}_p(A)\Big \}.$$
\end{thm}

\begin{proof} Observe that $$k_R(p)\otimes_R(A\otimes_RB)\cong
(k_R(p)\otimes_RA)\otimes_{k_R(p)}(k_R(p)\otimes_RB)$$ and that
$k_R(p)\otimes_RA$ is an AF-ring over the field $k_R(p)$. Then,
using \cite[Theorem 1.4]{BGK1}, we get, by Lemma 2.2(4),
$$\begin{array}{lll} \mbox
{dim}(k_R(p)\otimes_R(A\otimes_RB))&=&\mbox
{dim}\Big ((k_R(p)\otimes_RA)\otimes_{k_R(p)}(k_R(p)\otimes_RB)\Big )\\
&=&\mbox { sup}\Big \{D\Big (\mbox
{t.d.}\Big ((k_R(p)\otimes_RA)_{k_R(p)\otimes_RI}:k_R(p)\Big ),\\
&&ht(k_R(p)\otimes_RI),k_R(p)\otimes_RB\Big ):I\in\mbox {
Spec}_p(A)\Big \}\\
&=&\mbox { sup}\Big \{D\Big (\mbox
{t.d.}_p(A_I:R),ht_p(I),k_R(p)\otimes_RB\Big ):\\
&&I\in\mbox {
Spec}_p(A)\Big \}\\
&=&\mbox { sup}\Big \{D_p\Big (\mbox {t.d.}_p(A_I:R),ht_p(I),B\Big
):\\
&&I\in\mbox { Spec}_p(A)\Big \},\end{array}$$ as desired.
\end{proof}

The following corollaries compute the Krull dimension of tensor
products involving algebras whose (effective) Krull dimension is
zero. It is clear that if dim$(A)=0$, then for any nontrivial
$A$-algebra $C$, dim$_e^C(A)=0$. Also, in Example 6.6, we record the
existence of various cases of triplets of rings $(R,A,B)$ such that
either dim$_e^{A\otimes_RB}(R)=0$ or dim$_e^{A\otimes_RB}(A)=0$.

\begin{cor} Let $(R,A,B)$ be a triplet of rings such that $A$ is a $B$-fibred AF-ring and dim$_e^{A\otimes_RB}(R)=0$ (in particular, dim$(R)=0$).
Then,
$$\mbox {dim}(A\otimes_RB)=\mbox { sup}\Big \{D_p\Big (\mbox
{t.d.}_p(A_I:R),ht_p(I),B\Big ):p\in\mbox { Spec}_e^A(R)\cap\mbox
{Spec}_e^B(R)\mbox { and }$$ $$I\in\mbox { Spec}_p(A)\Big \}.$$
\end{cor}

\begin{proof} As dim$_e^{A\otimes_RB}(R)=0$, by Corollary 4.6(1),
$$\begin{array}{lll}\mbox {dim}(A\otimes_RB)&=&\mbox {f-dim}_R(A\otimes_RB)\\
&=&\mbox { sup}\{\mbox {dim}_p(A\otimes_RB):p\in\mbox {
Spec}_e^{A\otimes_RB}(R)\}.\end{array}$$Then, Theorem 6.2 completes
the proof.
\end{proof}

\begin{cor} Let $(R,A,B)$ be a triplet of rings such that $A$ is a $B$-fibred AF-ring and either dim$_e^{A}(R)=0$ or dim$_e^{B}(R)=0$.
Then,
$$\mbox {dim}(A\otimes_RB)=\mbox { sup}\Big \{D_p\Big (\mbox
{t.d.}_p(A_I:R),ht_p(I),B\Big ):p\in\mbox { Spec}_e^A(R)\cap\mbox
{Spec}_e^B(R)\mbox { and }$$ $$I\in\mbox { Spec}_p(A)\Big \}.$$
\end{cor}

\begin{proof} It is straightforward by Corollary 4.4 and Corollary 6.3.\end{proof}

We devote the following theorem to the case where one component of a
tensor product is a zero-dimensional ring.

\begin{thm} Let $(R,A,B)$ be a triplet of rings such that dim$_e^{A\otimes_RB}(A)=0$ (in particular, dim$(A)=0$). Then
$$\mbox {dim}(A\otimes_RB)=\mbox { sup}\Big
\{D_p\Big (\mbox {t.d.}_p(k_A(I):R),0,B\Big ):p\in\mbox {
Spec}_e^A(R)\cap\mbox { Spec}_e^B(R)\mbox { and }$$ $$I\in\mbox {
Spec}_p(A) \Big \}.$$
\end{thm}

\begin{proof} As dim$_e^{A\otimes_RB}(A)=0$, by Corollary 4.6(2),
$$\begin{array}{lll}\mbox {dim}(A\otimes_RB)&=&\mbox {f-dim}_A(A\otimes_RB)\\
&=&\mbox { sup}\{\mbox {dim}(k_A(I)\otimes_RB):I\in\mbox {
Spec}_e^{A\otimes_RB}(A)\}\end{array}.$$ Let $I\in\mbox {
Spec}_e^{A\otimes_RB}(A)$ and $p:=I\cap R$. Note that, by
Proposition 5.3, $k_A(I)$, being zero-dimensional, is a fibred
AF-ring over $R$. Also, as $I\in$ Spec$_e^{A\otimes_RB}(A)$, then
$p\in$ Spec$_e^{A\otimes_RB}(R)$. Therefore, Theorem 6.2 yields
$$\mbox {dim}_p(k_A(I)\otimes_RB)=D_p(\mbox
{t.d.}_p(A_I:R),ht_p(I),B).$$ Now, since $p\in$
Spec$_e^{A\otimes_RB}(R)$, by Theorem 4.2(3), any $J\in$ Spec$_p(A)$
is an effective prime ideal of $A$ with respect to $A\otimes_RB$.
Therefore, since dim$_e^{A\otimes_RB}(A)=0$, we get
$$ht_p(I)=ht(\frac I{pA})=0.$$  Furthermore, as
dim$_e^{k_A(p)}(R)=0$, we get, by Corollary 4.4,
dim$_e^{k_A(p)\otimes_RB}(R)=0$. It follows, by Corollary 4.6(1) and
as $k_A(I)\otimes_RB$ possesses only one fibre ring which is
$k_R(p)\otimes_R(k_A(I)\otimes_RB)$, that
$$\begin{array}{lll}\mbox {dim}(k_A(I)\otimes_RB)&=&\mbox
{f-dim}_R(k_A(I)\otimes_RB)\\
&=&\mbox { dim}_p(k_A(I)\otimes_RB)\\
&=&D_p(\mbox {t.d.}_p(A_I:R),ht_p(I),B)\\
&=&D_p(\mbox {t.d.}_p(A_I:R),0,B).\end{array}$$ Consequently,
$$\mbox {dim}(A\otimes_RB)=\mbox { sup}\{D_p(\mbox
{t.d.}_p(A_I:R),0,B):p\in\mbox { Spec}_e^{A\otimes_RB}(R)\mbox { and
}I\in\mbox { Spec}_p(A)\}$$ completing the proof as, by Theorem 3.3,
Spec$_e^{A\otimes_RB}(R)=$ Spec$_e^{A}(R)\cap$ Spec$_e^B(R)$.
\end{proof}

\begin{exa} 1) Let $R:=\mathbb{Z}$ and $A$, $B$ be rings
such that $(\mathbb{Z},A,B)$ is a triplet of rings. Let char$(A)=n$
and char$(B)=m$ such that $n\neq 0$ and $m\neq 0$. Then
Spec$_e^{A\otimes_{\mathbb{Z}}B}(\mathbb{Z})=\{p\mathbb{Z}:p$ is a
common prime divisor of $n$ and $m\}$ and thus
dim$_e^{A\otimes_{\mathbb{Z}}B}(\mathbb{Z})=0$.\\
2) Let $R:=\mathbb{Z}$. Let $n\geq 1$ be an integer. Let
$A=\mathbb{Z}_{p\mathbb{Z}}+X\mathbb{Q}[[X]]$ be a $D+m$
construction issued from the local ring
$\mathbb{Q}[[X]]=\mathbb{Q}+X\mathbb{Q}[[X]]$. Let $B$ be a
$\displaystyle {\frac {\mathbb{Z}}{n\mathbb{Z}}}$-algebra and let
$p$ be a prime divisor of $n$. Then,
Spec$(A)=\{(0),XQ[[X]],p\mathbb{Z}_{p\mathbb{Z}}+XQ[[X]]\}$ and
Spec$_e^{B}(\mathbb{Z})=\{q\mathbb{Z}:q$ is a positive prime divisor
of $n\}$. Therefore
Spec$_e^{A\otimes_{\mathbb{Z}}B}(\mathbb{Z})=\{p\mathbb{Z}\}$ and
Spec$_e^{A\otimes_{\mathbb{Z}}B}(A)=\{p\mathbb{Z}_{p\mathbb{Z}}+XQ[[X]]\}$.
It follows that dim$_e^{A\otimes_{\mathbb{Z}}B}(\mathbb{Z})=0$ and
dim$_e^{A\otimes_\mathbb{Z}B}(A)=0$.\end{exa}

Next, we deal with tensor products over the ring of integers
$\mathbb{Z}$. This allows us to answer a question rised by Jorge
Martinez on evaluating the Krull dimension of the tensor product
over $\mathbb{Z}$ of two rings one of which is a Boolean ring.

\begin{cor} {\it Let $(\mathbb{Z},A,B)$ be a triplet of rings such
that char($A)=:n\neq 0$. Assume that $A$ is a $B$-fibred AF-ring
over $\mathbb{Z}$. Let
$n=p_1^{\alpha_1}p_2^{\alpha_2}...p_r^{\alpha_r}$ be the
decomposition of $n$ into prime factors. Then}
$$\mbox {dim}(A\otimes_{\mathbb{Z}}B)=\mbox { sup}\Big \{D_{p_i}\Big (\mbox
{t.d.}_{p_i}(A_I:\mathbb{Z}),ht_{p_i}(I),B\Big ):i=1,2,\cdots,n
\mbox { and }I\in\mbox { Spec}_{p_i}(A)\Big \}.$$
\end{cor}

\begin{proof} Observe that, as char$(A)=n$, $\displaystyle {\frac
{\mathbb{Z}}{n\mathbb{Z}}}$ is identified to a subring of $A$ and
that Spec$\Big (\displaystyle {\frac {\mathbb{Z}}{n\mathbb{Z}}}\Big
)=\Big \{\displaystyle {\frac
{p_1\mathbb{Z}}{n\mathbb{Z}}},\displaystyle {\frac
{p_2\mathbb{Z}}{n\mathbb{Z}}},\cdots,\displaystyle {\frac
{p_r\mathbb{Z}}{n\mathbb{Z}}}\Big \}.$ Then, for each
$j=1,2,\cdots,r$, $\displaystyle {\frac
{p_j\mathbb{Z}}{n\mathbb{Z}}}$ is a minimal prime ideal of
$\displaystyle {\frac {\mathbb{Z}}{n\mathbb{Z}}}$ and thus there
exists $I_j\in$ Spec$(A)$ such that $I_j\cap\displaystyle {\frac
{\mathbb{Z}}{n\mathbb{Z}}}=\displaystyle {\frac
{p_j\mathbb{Z}}{n\mathbb{Z}}}$. Hence, for each $j=1,2,\cdots,r$,
there exists $I_j\in$ Spec$(A)$ such that $I_j\cap
\mathbb{Z}=p_j\mathbb{Z}$. Therefore Spec$_{e}^A(\mathbb{Z})=
\{p_1\mathbb{Z},p_2\mathbb{Z},\cdots,p_r\mathbb{Z}\}$). Thus
dim$_e^A(\mathbb{Z})=0$. Now, Corollary 6.4 completes the
proof.\end{proof}

We close with the following corollary which presents an answer to a
question of Jorge Martinez on evaluating the Krull dimension of the
tensor product over the ring of integers $\mathbb{Z}$ of two rings
one of which
is Boolean.\\

First, we record the following well known characteristics of Boolean
rings.

\begin{lem} Let $R$ be a Boolean ring. Then\\
1) $R$ is commutative.\\
2) char$(R)=2$.\\
3) dim$(R)=0.$\\
4) $\displaystyle {\frac Rp\cong \frac {\mathbb{Z}}{2\mathbb{Z}}}$
for each prime ideal $p$ of $R$.\\
5) $\displaystyle {\frac Rp}$ is an algebraic field extension of
$\displaystyle {\frac {\mathbb{Z}}{2\mathbb{Z}}}$ for each prime
ideal $p$ of $R$.\end{lem}

\begin{cor} Let $(\mathbb{Z},A,B)$ be a triplet of rings such that $A$ is a Boolean ring. Then
$$\begin{array}{lll}\mbox {dim}(A\otimes_{\mathbb{Z}}B)&=&\mbox {dim}\Big
(\displaystyle {\frac B{2B}}\Big ).\end{array}$$
\end{cor}

\begin{proof} Using the proof of Corollary 6.7, we
get Spec$_e^A(\mathbb{Z})=\{2\mathbb{Z}\}$. Also, as dim$(A)=0$, by
Proposition 5.3, $A$ is a fibred AF-ring over $\mathbb{Z}$.
Moreover,
$$\mbox {t.d.}_{2\mathbb{Z}}(A:\mathbb{Z})=\mbox { sup}\Big \{\mbox {t.d.}\Big (\displaystyle {\frac
AI:\frac {\mathbb{Z}}{2\mathbb{Z}}}\Big ):I\in\mbox { Spec}(A)\Big
\}=0$$ as $\displaystyle {\frac AI}$ is algebraic over
$\displaystyle {\frac {\mathbb{Z}}{2\mathbb{Z}}}$, by Lemma 6.8(5).
It follows, by Theorem 6.5, Lemma 6.1 and Lemma 6.8, that
$$\begin{array}{lll}\mbox {dim}(A\otimes_{\mathbb{Z}}B)&=&=D_{2\mathbb{Z}}(0,0,B)\\
&=&\mbox { sup}\Big \{ht\Big (\displaystyle {\frac J{2B}}\Big
):J\in\mbox { Spec}_{2\mathbb{Z}}(B)\Big \}\\
&=&\mbox { dim}\Big (\displaystyle {\frac B{2B}}\Big
)\end{array}$$completing the proof.\end{proof}

\end{document}